\documentclass[a4paper,12pt]{amsart}
\usepackage{latexsym,amssymb}
 \usepackage{amsthm}
\usepackage{epsfig}
\usepackage{amsmath}
\usepackage{amsbsy}
\usepackage{amscd}
\usepackage{url}
\usepackage{framed,empheq}
\usepackage{graphicx}
\usepackage{comment}
\usepackage{hyperref}

\newtheorem{thm}{\bf Theorem}

\newtheorem{lem}[thm]{\bf Lemma}
\newtheorem{claim}[thm]{\bf Claim}

\theoremstyle{remark}

\newcommand{\bb}[1]{\mathbb{#1}}
\newcommand{\bl}[1]{\boldsymbol{#1}}

\DeclareMathOperator{\Ker}{Ker}
 \newcommand{\boxedeq}[2]{\begin{empheq}[box={\fboxsep=6pt\fbox}]{align}\label{#1}#2\end{empheq}}

\title{A distribution on triples with maximum entropy marginal}
\author{Sergey Norin}  \address{Department of Mathematics and Statistics, McGill University}
\thanks{The author is supported by an NSERC Discovery grant. }

\begin{document}

	\begin{abstract} We construct an  $S_3$-symmetric  probability distribution on $\{(a,b,c) \in \bb{Z}_{\geq 0}^3 \: : \: a+b+c =n \}$ such that its marginal achieves the maximum entropy among all probability distributions on $\{0,1,\ldots,n\}$ with mean  $n/3$.
		Existence of such a distribution verifies a conjecture of Kleinberg, Sawin and Speyer~\cite{KSS}, which is motivated by the study of sum-free sets.    
	\end{abstract}
	\maketitle
	
\section{Introduction}

The recent breakthrough by Croot, Lev and Pach~\cite{CLP} and the subsequent solution of the cap-set problem by Ellenberg and Gijswijt~\cite{EG} led to a dramatic improvement of known upper bounds on the size of maximum sum-free sets in powers of finite groups. Blasiak et al.~\cite{BCCGU} extended Ellenberg-Gijswijt result to multi-colored sum-free sets. Even more recently Kleinberg, Sawin, and Speyer~\cite{KSS} established upper bounds for the multi-colored version which are  essentially tight. Let us state the main result of~\cite{KSS}, which motivates our work.

Let $p$ be a prime. A \emph{tri-colored sum-free set in $\bb{F}_p^n$} is a collection of triples $\{(x_i,y_i,z_i)\}_{i=1}^m$ of elements of   $\bb{F}_p^n$ such that $x_i + y_j +z_k =0$ if and only if $i=j=k$. Kleinberg, Sawin, and Speyer establish an upper bound $m \leq e^{\gamma_p n}$ on the size of a tri-colored sum-free set in $\bb{F}_p^n$, where $\gamma_p$ is as follows.

The \emph{entropy} of a probability distribution $\mu$ on a finite set $I$ is defined as $$\eta(\mu) = \sum_{i \in I}\mu(i)\log\mu(i),$$
where we interpret $0 log 0$ as 0.
Let $T = \{(a,b,c) \in \bb{Z}_{\geq 0}^3 \: : \: a+b+c =p-1 \}$. Let $\pi$ be an $S_3$-symmetric probability distribution on $T$, and let $\mu(\pi)$ be the marginal probability distribution of $\pi$ on $[0,p-1]$\footnote{We denote the set of consecutive integers $\{n,n+1,\ldots, n+k\}$ by $[n,n+k]$}  corresponding to the first coordinate.
(As $\pi$ is $S_3$-symmetric the choice of a coordinate is irrelevant.) Let $\gamma_p$ be the maximum entropy of $\mu(\pi)$  over $S_3$-symmetric probability distributions $\pi$ on $T$. 

\begin{thm}[Kleinberg, Sawin, and Speyer~\cite{KSS}\footnote{In the published version of \cite{KSS}, the upper bound part of the statement of Theorem~\ref{t:KSS}, as well as the examples for $n \leq 25$ referenced later, were removed, as the proofs of Theorem~\ref{t:mainz} in an earlier version of this article, and independent work of Pebody~\cite{Pebody}, showed that they were unnecessary, but they are still available in the referenced arxiv version.}]
\label{t:KSS}
All tri-colored sum-free sets  in $\bb{F}_p^n$ have size at most $e^{\gamma_p n}$. Moreover, there exist  tri-colored sum-free sets  in $\bb{F}_p^n$ of size at least 
 $e^{\gamma_pn - o(n)}$.
\end{thm}

Every marginal of a symmetric probability distribution on $T$ has mean $(p-1)/3$.  Therefore, $\gamma_p$ is at most the maximum entropy of a probability distribution  on $[0,p-1]$ with this mean.  The main result of this paper, which was independently obtained by Pebody~\cite{Pebody}, gives a proof of ~\cite[Theorem 4]{KSS} showing that the equality holds.

\begin{thm}\label{t:mainz} For every $n \geq 1$
there exists an $S_3$-symmetric probability distribution  $\pi$ on $\{(a,b,c) \in \bb{Z}_{\geq 0}^3 \: : \: a+b+c =n\}$ such that $\mu(\pi)$  achieves the maximum entropy among probability distributions on $[0,n]$ with  mean $n/3$.
\end{thm}

While the definition of $\gamma_p$ above already implies that it is a computable constant, Theorem~\ref{t:mainz} provides a much simpler description. As noted in~\cite{KSS}, a direct calculation shows that if  $\mu$  has the maximum entropy among probability distributions on $[0,n]$ with  mean $n/3$ then
\begin{equation}\label{e:mu}
\mu(i)=\frac{\rho^i}{1+\rho + \ldots + \rho^n}
\end{equation}
where  $\rho$ is the unique positive real solution to the equation 
\begin{equation}\label{e:rho}
\sum_{i=0}^n i\rho^i = \frac{n}{3}\sum_{i=0}^n \rho^i.
\end{equation}
Further, Theorem~\ref{t:mainz} confirms that the upper bound in Theorem~\ref{t:KSS} coincides with the bounds established for sum-free sets in~\cite{EG}  and three-colored  sum-free sets in~\cite{BCCGU}. The value of $\gamma_p$ is also of interest as it appears in the tight bound for the arithmetic
triangle removal lemma of Fox and Lov\'asz~\cite{FL}.

We construct a distribution  $\pi$ satisfying Theorem~\ref{t:mainz} explicitly. Examples of the distributions satisfying  Theorem~\ref{t:mainz} for $n \leq 25$ are provided in~\cite{KSS}. Based on these examples and additional experimentation, we construct a  simple $S_3$-symmetric function of $\{(a,b,c) \in \bb{Z}_{\geq 0}^3 \: : \: a+b+c =n\}$ with the marginal given by (\ref{e:mu}).  This construction is presented in 
Section~\ref{s:notation} along with the necessary notation. Unfortunately, the constructed function fails to be non-negative for  $n \geq 28$. 
In Section~\ref{s:technical} we modify via a sequence of ``local" changes  to establish Theorem~\ref{t:mainz} in full generality. 

\section{Notation and the first attempt}\label{s:notation}
Fix positive integer $n$ for the remainder of the paper. Let $T = \{(a,b,c) \in \bb{Z}_{\geq 0}^3 \: : \: a+b+c =n \}$.
The probabilistic distributions we are interested in form a polytope in $\bb{R}^T$, and vectors in $ \bb{R}^T $ will be the main object of study in the remainder of the paper. We use  the convention $\bl{v}=(v_{abc})_{(a,b,c) \in T}$, that is we denote by  $v_{abc}$ the component of the vector $v$ corresponding to a triple $(a,b,c) \in T$. Let $\{\bl{e}(a,b,c)\}_{(a,b,c) \in T}$ be the standard basis of $ \bb{R}^T$.
We say that a vector $\bl{v} \in \bb{R}^T$  is \emph{symmetric} if $v_{i_1i_2i_3}=v_{i_{\sigma(1)}i_{\sigma(2)}i_{\sigma(3)}}$ for every permutation $\sigma \in S_3$. Let $W \subseteq \bb{R}^T$ be the vector space of symmetric vectors. For $(i_1,i_2,i_3)  \in T$ define $$\bl{s}(i_1,i_2,i_3) = \sum_{\sigma \in S_3}\bl{e}(i_{\sigma(1)},i_{\sigma(2)},i_{\sigma(3)}).$$
The set $\{\bl{s}(a,b,c)\: | \:(a,b,c) \in T, a \leq b \leq c\}$ forms a convenient basis of $W$.
For $\bl{v} \in \bb{R}^T$  and $a \in [0,n]$ define $$\mu_a(\bl{v})=\sum_{i=0}^{n-a}v_{ai(n-a-i)},$$ 
and let $\mu: \bb{R}^T \to \bb{R}^{[0,n]}$ be defined by $$\mu(\bl{v})=(\mu_0(\bl{v}),\mu_1(\bl{v}), \ldots,\mu_n(\bl{v})).$$
Note that importantly
\begin{equation}\label{e:mus}
\mu_i(\bl{s}(a,b,c))=2(\delta_{ia}+\delta_{ib}+\delta_{ic}),
\end{equation}
for $i \in [0,n]$ and $(a,b,c) \in T$, where $\delta$ is the Kronecker delta.

Let $\rho$ de defined by (\ref{e:rho}).
Clearly, $\rho<1$.
Define $$\bl{r} = (n,n\rho,\ldots,n\rho^n) \in \bb{R}^{[0,n]}.$$
We say that $\bl{v} \in \bb{R}^T$ is \emph{$\rho$-marginal} if $\mu(\bl{v})=\bl{r}$. Let $R$ denote the set of symmetric, $\rho$-marginal vectors in $ \bb{R}^T$. Note that $R$ is an affine space and $R = \bl{v}+ (\Ker(\mu) \cap W)$ for any $\bl{v} \in R$.

It is easy to see that the next theorem is a reformulation of Theorem~\ref{t:main} using the introduced terminology. (Note that for convenience we scaled the target marginal by $n(1+\rho+\ldots +\rho^n)$.)

\begin{thm}\label{t:main}
There exists a non-negative vector $\bl{\pi} \in R$.
\end{thm}

We construct an explicit, albeit not particularly elegant vector $\bl{\pi}$ satisfying Theorem~\ref{t:main}. As a first step in this section we construct an auxiliary vector $\bl{\beta} \in R$, which has a compact description and will be the  starting point of the general construction. In Section~\ref{s:technical} we finish the construction of $\bl{\pi}$ by defining a generating set of $\Ker(\mu)$ consisting of vectors with small support and adding an appropriate linear combination of these vectors to $\bl{\beta}$.

We now define $\bl{\beta}$. Let
\begin{equation}\label{e:pidef}
\beta_{abc}= \rho^{a}-\rho^{n-a}+\rho^{b}-\rho^{n-b}+\rho^{c}-\rho^{n-c},
\end{equation} for $(a,b,c) \in T$, $a,b,c \geq 1$, let 
\begin{equation}\label{e:zerodef}
 \beta_{a0(n-a)}=\sum_{i=1}^{a-1}(-\rho^{i}+\rho^{n-i})+ \frac{a-1}{2}\rho^a + \frac{n-a+1}{2}\rho^{n-a} 
 \end{equation}
 for $1 \leq a \leq n/2$,
let  
\begin{equation}\label{e:ndef}
 \beta_{00n}=n\rho^n,
 \end{equation}
and define the components of $\bl{\beta}$ for the remaining triples in $T$ by symmetry, so that  $\bl{\beta}$ is symmetric.

\begin{lem}\label{l:beta1} The vector $\bl{\beta}$ is $\rho$-marginal.
\end{lem}

\begin{proof}
First, let us note that 
the identity (\ref{e:zerodef}) also holds for $n/2 < a \leq n-1$. Indeed, for such $i$ we have
\begin{align*}
 \beta_{a0(n-a)} &=\sum_{i=1}^{n-a-1}(-\rho^{i}+\rho^{n-i})+ \frac{n-a-1}{2}\rho^{n-a} + \frac{a+1}{2}\rho^a \\
  &=\sum_{i=1}^{n-a-1}(-\rho^{i}+\rho^{n-i}) + \sum_{i=n-a}^{a}(-\rho^{i}+\rho^{n-i}) 
  +\frac{n-a-1}{2}\rho^{n-a}   + \frac{a+1}{2}\rho^a \\
 &=\sum_{i=1}^{a-1}(-\rho^{i}+\rho^{n-i}) -\rho^{a}+\rho^{n-a} +\frac{n-a-1}{2}\rho^{n-a}   + \frac{a+1}{2}\rho^a \\&=\sum_{i=1}^{a-1}(-\rho^{i}+\rho^{n-i})+ \frac{a-1}{2}\rho^a + \frac{n-a+1}{2}\rho^{n-a},  
\end{align*}
as desired.

Now we are ready to verify that 
$
\mu_a(\bl{\beta})=n\rho^a
$
for every $a \in [0,n]$. 
We have $
\mu_n(\bl{\beta})=n\rho^n
$ by (\ref{e:ndef}). For 
$1 \leq a \leq n-1$,  we have
 	\begin{align*}
    \mu_a(\bl{\beta})& = \sum_{i=0}^{n-a} \beta_{ai(n-a-i)} = 
 	\sum_{i=1}^{n-a-1} (\rho^{a}-\rho^{n-a}+\rho^{i}-\rho^{n-i}+\rho^{n-a-i}-\rho^{a+i}) \\&
 	+ 2\sum_{i=1}^{a-1}(-\rho^{i}+\rho^{n-i}) + (a-1)\rho^a +(n-a+1)\rho^{n-a} \\ = &	(n-a-1)(\rho^{a}-\rho^{n-a}) + 	2\sum_{i=1}^{n-a-1}(\rho^{i}-\rho^{n-i}) - 2\sum_{i=1}^{a-1}(\rho^{i}-\rho^{n-i}) \\&+ (a-1)\rho^a +(n-a+1)\rho^{n-a}  \\ =&
 	(n-2)\rho^a + 2\rho^{n-a} + 2(\rho^a - \rho^{n-a}) = n\rho^a,
 	\end{align*} 
 as desired.
 Finally, we have
 \begin{align*}
 \mu_0(\bl{\beta}) &=\sum_{i=0}^{n} \beta_{0i(n-i)} \\ &= 2n\rho^n +
 \sum_{i=1}^{n-1} \left(\sum_{j=1}^{i-1}(-\rho^{j}+\rho^{n-j})+ \frac{i-1}{2}\rho^i + \frac{n-i+1}{2}\rho^{n-i} \right) \\ &= 2n\rho^n + \sum_{i=1}^{n-1}(-(n-i-1) + (i-1)+i)\rho^i \\&= \sum_{i=1}^{n}(3i-n)\rho^i  = n + \sum_{i=0}^{n}(3i-n)\rho^i = n,
 \end{align*}
 where the last equality uses (\ref{e:rho}).
\end{proof}

Upon cursory examination $\bl{\beta}$ appears to be a promising candidate for a non-negative vector in $R$. In fact,  $\bl{\beta}$ is the only vector in $R$ for $n \leq 5$. In general, it is easy to see that $\beta_{abc} \geq 0$ for  $a,b,c \geq 1$. Unfortunately, $\bl{\beta}$ is non-negative only for $n \leq 27$, whereas $\beta_{0\lfloor n/2\rfloor\lceil n/2\rceil} <0$ for $n \geq 28$. 
Thus we have to modify $\bl{\beta}$ by adding to it an appropriate vector in $\Ker(\mu)$. This is the goal of the next, somewhat technical section.

\section{Flattening $\bl{\beta}$.}\label{s:technical} 

Given $(a,b,c) \in T$ and $x,y \in \bb{N}$, such that $b \geq x$, $c \geq x+y$  we define a vector 
\begin{align*}
\bl{m}&^{x,y}(a,b,c)=-\bl{s}(a,b,c)+\bl{s}(a+x,b-x,c)\\ &-\bl{s}(a+x+y,b-x,c-y) +\bl{s}(a+x+y,b,c-x-y)\\ &-\bl{s}(a+x,b+y,c-x-y)+\bl{s}(a,b+y,c-y).
\end{align*}

\begin{claim}\label{c:move}
We have $\bl{m}^{x,y}(a,b,c) \in \Ker(\mu) \cap W$ for all  $(a,b,c) \in T$ and $x,y \in \bb{N}$, such that $b \geq x$, $c \geq x+y$. 
\end{claim}	

\begin{proof} Clearly, $\bl{m}^{x,y}(a,b,c) \in W$. Therefore we
 only need to show that $\mu_i(\bl{m}^{x,y}(a,b,c))=0$ for every $i \in [0,n]$. This follows immediately from (\ref{e:mus}), as each  Kronecker deltas $\delta_{ij}$ for $j \in \{a,a+x,a+x+y,b,b-x,b+y,c, c-y,c-x-y\}$ will appear the same number of times with positive and negative signs in the expansion of $\mu_i(\bl{m}^{x,y}(a,b,c))$ using (\ref{e:mus}).
\end{proof}

We think of vectors $\bl{m}^{x,y}(a,b,c)$ as directions, along which a  vector in $R$ can be shifted to obtain another  vector in $R$ differing from the original only in a few coordinates. In the remainder of the section we describe a collection of such shifts which transform $\bl{\beta}$ into a non-negative vector.

We will only use a subset of  the vectors defined above of the following form. 
 $\bl{m}(b \to a) = \bl{m}^{\lceil b/2 \rceil,a-b}(0,b,n-b)$ for $2 \leq b < a \leq n/2$. That is 
\begin{align}\label{e:move2}
\bl{m}&(b \to a)=-\bl{s}(0,b,n-b)+\bl{s}(\lceil b/2 \rceil,\lfloor b/2 \rfloor,n-b)\notag\\ &-\bl{s}(a-\lfloor b/2 \rfloor ,\lfloor b/2 \rfloor,n-a)+\bl{s}(a-\lfloor b/2 \rfloor ,b,n-a-\lceil b/2 \rceil) \notag\\ &-\bl{s}(\lceil b/2 \rceil,a,n-a-\lceil b/2 \rceil)+\bl{s}(0,a,n-a).
\end{align}

We obtain $\bl{\pi}$ from $\bl{\beta}$ by adding to it a linear combination of vectors $\bl{m}(b \to a)$. 
The coefficients of these vectors, which we define next are chosen so that, in particular, $\pi_{0a(n-a)}=\pi_{0b(n-b)}$ for $2 \leq a,b \leq n-2$. (Hence we think of the construction of $\bl{\pi}$ as ``flattening $\bl{\beta}$".) Denote $\beta_{0i(n-i)}$ by $z_i$ for brevity, and let 
\begin{align*}c_{b} = \frac{2}{n+1-2b} \left(z_b - \frac{1}{n-1-2b}\sum_{i=b+1}^{n-b-1}z_i\right) \\=\frac{\sum_ {i=b}^{n-b} z_i}{n+1-2b} - \frac{\sum_{i=b+1}^{n-b-1} z_i}{n-1-2b}.\end{align*}
Let $c_{ba}=c_b$ for $2 \leq b < a < n/2$, and let $c_{b(n/2)}=\frac{c_b}{2}$ for even $n$.

We are now ready to define $\bl{\pi}$:
\boxedeq{e:pi}{\bl{\pi} = \bl{\beta}+ \sum_{2 \leq b \leq n/2 - 1}\;\sum_{b+1 \leq a  \leq n/2}c_{ba}\bl{m}(b \to a)}
In the remainder of the section we show that $\bl{\pi}$ satisfies Theorem~\ref{t:main}. It follows from Claim~\ref{c:move} that $\bl{\pi} \in R$. Thus it remains to show that $\pi_{abc}\geq 0$ for $(a,b,c) \in T$. This is accomplished in the following series of claims.
 
\begin{claim}\label{c:zeros} For $2 \leq a \leq n/2$ we have 
	\begin{equation}\label{e:zeros}
	\pi_{0a(n-a)} = \frac{n}{n-2}\left(1 - \rho^n -\rho^{n-1}/2\right) \geq 0
	\end{equation}
\end{claim}	

\begin{proof}
	We have $\pi_{00n}=n\rho^n$, $\pi_{01n}=n\rho^{n-1}/2$, and $\mu_0(\bl{\pi})=n$, as $\bl{\pi}$ is $\rho$-marginal. Thus to establish (\ref{e:zeros}) it suffices to verify that $\pi_{0an-a}=\pi_{0bn-b}$ for $2 \leq a,b \leq n-2$. Define $$\bl{\pi^i}=
	\bl{\beta}+ \sum_{i \leq b \leq n/2 - 1}\;\sum_{b+1 \leq a  \leq n/2}c_{ba}\bl{m}(b \to a)$$
	We will show by induction on $n/2-b$ that 	\begin{equation}\label{e:induct}\pi^b_{0a(n-a)} = \frac{\sum_{i=b}^{n-b}z_i}{n+1-2b}
		\end{equation}
	for every $b \leq a \leq n-b$. The identity (\ref{e:induct}) for $b=2$ implies the claim.
	 
	The base case $b = \lfloor n/2 \rfloor$ is trivial. For the induction step and $b<a<n-b$  
	\begin{align*}\pi&^b_{0a(n-a)} = \pi^{b+1}_{0a(n-a)}+c_b \\&=\frac{\sum_{i=b}^{n-b}z_i}{n-1-2b}+\left(\frac{\sum_ {i=b}^{n-b} z_i}{n+1-2b} - \frac{\sum_{i=b+1}^{n-b-1} z_i}{n-1-2b}\right) \\& 
	 = \frac{\sum_{i=b}^{n-b}z_i}{n+1-2b},
\end{align*}
	 as desired, as $z_b=z_{n-b}$. 
	 Finally, \begin{align*}\pi&^b_{0b(n-b)} = \pi^{b+1}_{0b(n-b)}-\frac{n-1-2b}{2}c_b \\&=z_b - \frac{n-1-2b}{n+1-2b}\left(z_b - \frac{1}{n-1-2b}\sum_{i=b+1}^{n-b-1}z_i\right) \\& 
	 = \frac{\sum_{i=b}^{n-b}z_i}{n+1-2b},
	 \end{align*}
	 finishing the proof of the identity in (\ref{e:induct}). 
	 
	 It remains to show that $\rho^n +\rho^{n-1}/2 \leq 1$. Using (\ref{e:rho}), we have
	 \begin{align*}
	 \frac{n(n+1)}{3}&\left(\rho^n +\frac{\rho^{n-1}}{2}\right) \leq n\rho^n + \frac{n(n-1)}{2}\rho^{n-1} \\ &\leq \sum_{i=0}^{n}i\rho^i =\frac{n}{3}\sum_{i=0}^{n}\rho^i \leq \frac{n(n+1)}{3}, 
	 \end{align*}
	 implying the desired inequality.
\end{proof}

 We now proceed to establish the estimates which will allow us to prove the non-negativity of $\bl{\pi}$. We start with a couple of  indirect bounds on  $\rho$. 
 
  \begin{claim}\label{c:lower} We have  	
  	\begin{equation}\label{e:term}
  	\rho^{n/2+1} \geq \frac{2}{3e} .
  	\end{equation}
  \end{claim}
   \begin{proof}
   	It can be verified by a computer that $\rho^{n/2+1} \geq 1/3$  for $n \leq 15$. Thus (\ref{e:term}) holds for $n \leq 15$, and we assume $n \geq 16$.
   	
   Let $\alpha = (n+2)(1-\rho)/\rho$. Then $\rho=(n+2)/(n+2+\alpha)$, and 	\begin{equation}\label{e:alpha} 
   \rho^{n+1} \geq \rho^{n+2}=\frac{1}{(1+\alpha/(n+2))^{n+2}} \geq e^{-\alpha}.\end{equation}
   We claim that  $\alpha \leq  2\log(3e/2)$. Note that by  (\ref{e:alpha})  this claim implies (\ref{e:term}).

  We start the proof of the claim by multiplying both sides of (\ref{e:rho}) by $(1-\rho)^2$, expanding and rearranging terms to obtain 
    $$  (n+3)\rho-n - \rho^{n+1}((2n+3) - 2n\rho)=0.$$
Using (\ref{e:alpha}) to bound $\rho^{n+1}$ and otherwise expressing $\rho$ in terms of $\alpha$,  we get $$  \frac{(n+3)(n+2)}{n+2+\alpha} - n -  e^{-\alpha}\left(2n+3 - \frac{2n(n+2)}{n+2+\alpha}\right)  \geq 0. $$
Multiplying the above inequality by $n+2+\alpha$, we obtain
$$3n+6 - n\alpha -  e^{-\alpha}(3n+6+(2n+3)\alpha)  \geq 0.$$
Let $f(x,n)=3n+6 - nx -  e^{-x}(3n+6+(2n+3)x)$. 
We have $f(2\log(3e/2),16)=-0.14055..,$ and  $\frac{\partial f}{\partial x}(x,n)$ and  $\frac{\partial f}{\partial n }(x,n)$ are easily seen to be negative for  $x\geq 2\log(3e/2)$ and $n \geq 16$. Thus $f(x,n) < 0$ for all  $x\geq 2\log(3e/2)$ and $n \geq 16$. As $f(\alpha,n) \geq 0$, we have $\alpha < 2\log(3e/2)$, as desired.
\end{proof} 
  
  Define $$\delta=\frac{1-\rho}{e\rho}.$$ 
 
  \begin{claim}\label{c:max} We have  	$$
  	(i+1)(1-\rho)^2\rho^i \leq\delta.
  $$
  	for all $i \geq 0$.
  \end{claim}
  \begin{proof}
  	By the AM-GM inequality we have $$\left(\frac{\rho}{i+1}\right)^{i+1}(1-\rho) \leq \left(\frac{1}{i+2}\right)^{i+2}.$$
  	Therefore $$	(i+1)(1-\rho)^2\rho^i \leq \frac{(1-\rho)}{\rho}\left(\frac{i+1}{i+2}\right)^{i+2} \leq  \frac{(1-\rho)}{\rho} \frac{1}{e}=\delta,$$
  	as desired. 
  \end{proof}  	
  
  Next we estimate $c_b$. Define $\Delta_b = z_b -z_{b+1}$. Direct calculation shows that
  $$
  \Delta_b=\frac{1}{2}\left((b+1)\rho^b-b\rho^{b+1} + (n-b-1)\rho^{n-b}-(n-b)\rho^{n-b-1} \right).
  $$

   \begin{claim}\label{c:Delta1} We have  $$\Delta_{i+1} \leq \Delta_i \leq \Delta_{i+1}+ \delta$$ for all $1 \leq i \leq n-1$. 
   \end{claim}
  \begin{proof}
  We have
  \begin{align*}
  2&(\Delta_{i} - \Delta_{i+1})= (i+1)\rho^i - (2i+2)\rho^{i+1}+(i+1)\rho^{i+2} \\ &
  +(n-i-1)\rho^{n-i-2} - 2(n-i-1)\rho^{n-i-1} + (n-i-1)\rho^{n-i}\\&= (1-\rho)^2((i+1)\rho^i + (n-i-1)\rho^{n-i-2}). 
  \end{align*}
  The last term is clearly non-negative, and it is at most $2 \delta$ by Claim~\ref{c:max}. Thus the claim holds.
  \end{proof}
  
   \begin{claim}\label{c:Delta2} We have $$0 \leq \Delta_{i} \leq \frac{\delta (n-1-2i)}{2} $$ for all integers $1 \leq i \leq (n-1)/2$.
   	\end{claim}
   	  \begin{proof}
   	  	As $z_i = z_{n-i}$ we have $\Delta_{i}=-\Delta_{n-1-i}$ for all $i$. Thus, if $n$ is odd, we have $\Delta_{(n-1)/2}=0$. For even $n$, we have $\Delta_{n/2}=-\Delta_{n/2-1}$, and $\Delta_{n/2} \leq \Delta_{n/2-1} \leq \Delta_{n/2} + \delta $ by Claim~\ref{c:Delta1}. Thus $0 \leq \Delta_{n/2-1} \leq \delta/2$. This establishes the claim for $i \in \{(n-1)/2,n/2-1\}$. 
   	  	
   	  	The claim for general $i$ follows directly from Claim~\ref{c:Delta1} by induction on $\lfloor n /2 \rfloor -i$, with the result of the preceding paragraph used as the base case.
   	  \end{proof}

   \begin{claim}\label{c:c1} We have \begin{equation}
   	\label{e:c} c_b \leq \frac{\delta(n-2b)}{6}
   	\end{equation} for every positive integer $2 \leq b \leq n/2 -1$.
   \end{claim}
   \begin{proof}
   	As in Claim~\ref{c:Delta2} the proof is by induction on  $\lfloor n /2 \rfloor -b$, and the base case is $b = \lfloor n /2 \rfloor -1$.  
   	
   	Suppose first that $n$ is even.  Then $n=2b+2$ in the base case, and $c_b=2\Delta_b/3$ by definition. As $0 \leq \Delta_b \leq \delta/{2} $ by Claim~\ref{c:Delta2}, (\ref{e:c}) holds. If $n$ is odd, then $n=2b+3$, $c_b = \Delta_b/2$, and $0 \leq \Delta_b \leq \delta $ by Claim~\ref{c:Delta2}, implying that (\ref{e:c}) once again holds. This finishes the proof of the base case.
   	
   	For the induction step note that \begin{align*} &\frac{(n+1-2b)(n-1-2b)c_b}{2} \\&= (n-1-2b)z_{b} -  -\sum_{i=b+1}^{n-b-1}z_i  \\
   &=	(n-1-2b)\Delta_b + (n-1-2b)z_{b+1} - \sum_{i=b+1}^{n-b-1}z_i \\ 
    &= (n-1-2b)\Delta_b+ (n-1-2(b+1))z_{b+1}- \sum_{i=b+2}^{n-b-2}z_i \\
    &=(n-1-2b)\Delta_b + \frac{(n-1-2b)(n-3-2b)c_{b+1}}{2}
    \end{align*}
    Thus $$
   c_b = \frac{2\Delta_b+(n-3-2b)c_{b+1}}{(n+1-2b)}.
  $$ 
   Using the bounds on $\Delta_b$ from Claim~\ref{c:Delta2} and the induction hypothesis applied to $c_{b+1}$ we obtain
   $$ 0 \leq c_b \leq \frac{\delta(n-1-2b) + \delta(n-3-2b)(n-2-2b)/6}{n+1-2b} = \frac{\delta(n-2b)}{6},$$
   as desired.
   \end{proof}
   
     \begin{claim}\label{c:c2} We have 	$$\rho^b-\rho^{n-b} \geq 4c_b$$ for all positive integers $2 \leq b \leq n/2 -1$.
     \end{claim}
     \begin{proof} Let $$f(x)=\rho^x-\rho^{n-x}-\frac{2\delta(n-2x)}{3}.$$
     By Claim~\ref{c:c1} it suffices to show that $f(x) \geq 0$ for all $x \leq n/2$. As $f(n/2)=0$, and $f''(x) \geq 0$ for $x \leq n/2$, it is enough to verify that $f'(n/2)\leq 0$, i.e.
    $$
     -2\rho^{n/2}\log{\rho} \geq \frac{4}{3}\delta = \frac{4(1-\rho)}{3e\rho}
     $$
     As $-\log\rho \geq 1-\rho$, the above is implied by Claim~\ref{c:lower}.
     \end{proof}
  
The next claim finishes the proof of Theorem~\ref{t:main}.

 \begin{claim}\label{c:main} We have $\pi_{xyz} \geq 0$ for $(x,y,z) \in T$, $x,y,z \geq 1$.
 \end{claim}

  \begin{proof} Assume that  $x \leq y \leq z$. Suppose that the component corresponding to $(x,y,z)$ is negative in $\bl{m}(b \to a)$ for some $2 \leq b < a \leq n/2$. Direct examination of (\ref{e:move2}) shows that, if $z > n/2$, then \begin{description}
  		\item[(N1)] either $b \in \{2x,2x+1\}$ and $a=x+y$, in which case $(x,y,z)=(\lfloor b/2 \rfloor,a-\lfloor b/2 \rfloor ,n-a)$, or 
  		\item[(N2)]  $b \in \{2x,2x-1\}$, $a=y$, in which case $(x,y,z)=(\lceil b/2 \rceil,a,n-a-\lceil b/2 \rceil)$, 
  	\end{description}
  	If $z<n/2$ then  \begin{description}
  		\item[(N2)] either $b \in \{2x,2x-1\}$ and  $a=y$,  $(x,y,z)=(\lceil b/2 \rceil,a,n-a-\lceil b/2 \rceil)$ as before,  or
  		\item[(N3)]  $b \in \{2x,2x-1\}$, $a=z$, in which case $(x,y,z)=(\lceil b/2 \rceil,n-a-\lceil b/2 \rceil,a)$.
  	\end{description}
 If $z=n/2$ then any of the above cases can potentially occur, but in cases (N1) and (N3) the component of $\bl{m}(b \to a)$ corresponding to $(x,y,z)$ is equal to $c_b/2$ rather than $c_b$.
 
  Suppose first that $x<y<z$.
By the above analysis, the total negative contribution to $\pi_{xyz}$ of vectors $\bl{m}(b \to a)$ for $2 \leq b < a \leq n/2$ is at most $4\max\{c_{2x},c_{2x+1},c_{2x-1}\} \leq  \rho^{x}-\rho^{n-x}$, where the last inequality holds by Claim~\ref{c:c2}. Thus
 $$\pi_{xyz} \geq  \beta_{xyz} - \rho^{x}-\rho^{n-x} \geq (\rho^z - \rho^{n-y}) + (\rho^y - \rho^{n-z}) > 0,$$
 as desired. 
 
 Finally, suppose that $x \leq y \leq z$, and one of these inequalities is non-strict.
 Then a vector $\bl{m}(b \to a)$  can only contribute negatively to $\pi_{xyz}$ if $x<y=z$ and $(x,y,z)=(\lceil b/2 \rceil,a,n-a-\lceil b/2 \rceil)$. Note that in this case the component of  $\bl{m}(b \to a)$ corresponding to $(x,y,z)$ is equal to $2c_b$, rather than $c_b$, but the bound  $4\max\{c_{2x},c_{2x+1},c_{2x-1}\}$ on the total negative contribution established in the previous case still holds.
  	\end{proof} 
  	
\section*{Acknowledgments}
We would like to thank  the anonymous referee and Lisa Sauermann for pointing out an error in the estimates used in the proof of Theorem~\ref{t:main} in an earlier versson of this paper.
  	
\bibliographystyle{plain}
\bibliography{snorin.bib}

\end{document}